\documentclass[11pt,a4paper]{amsart}
\usepackage{braket,stmaryrd}
\title{Implicitly definable generalized quantifiers}
\author{Fredrik Engström}
\date{\today}

\newtheorem*{thm}{Theorem}
\newtheorem*{lem}{Lemma}
\theoremstyle{definition}
\newtheorem*{defin}{Definition}
\newtheorem*{example}{Example}
\newtheorem*{que}{Question}

\newcommand\Q{\mathsf Q}
\newcommand\imp\rightarrow
\newcommand\Pred{{\text{Pred}}}
\newcommand\str[1]{\mathcal #1}

\begin{document}
\begin{abstract}
We give a new elementary proof of the main theorem of \cite{feferman2012quantifiers}: Quantifiers implicitly definable in pure second-order logic equipped with Henkin semantics implies are (explicitly) definable in first-order logic.
\end{abstract}

\maketitle

%\section*{Dear Christian,}

Dear Chris,

\vspace{2em}

You have recently, together with Jörgen Sjögren in \cite{bennet2011mathematical,bennet2013philosophy}, proposed that concepts should be considered mathematical iff they have a unique explication (in the sense of Carnap). In this short paper I will discuss a somewhat similar idea: Symbols (or operations) should be considered logical iff they can be defined uniquely by inference rules. This idea can be traced back at least to Belnap's response \cite{belnap1962tonk} to Prior's argument in his famous Tonk paper \cite{prior1960runabout}; but let us start from the beginning.

Logic considers the \emph{form}, or to be more precise, the \emph{logical form}, of sentences and arguments. To determine this form we need to distinguish the \emph{logical constants} from the other symbols; we need to characterize the logical constants, answering the fundamental question: 
\begin{quote}
Which symbols should be considered \emph{logical}?
\end{quote}
For example, what is the reason for us think of $\forall$ and $\land$ as logical but not $P$ and $c$ in the formula $\forall x (P(x) \lor x=c)$?

Several answers to this questions has been proposed; we shall concentrate on a proof theoretic idea (or rather a model theoretic take on the proof theoretic idea): Symbols that can be defined by inference rules are logical. Belnap, in \cite{belnap1962tonk}, proposed that ``defined'' should be understood as being governed by \emph{conservative} inference rules that define the symbol \emph{uniquely}. Regarding uniqueness he writes:
\begin{quote}
It seems rather odd to say we have defined plonk unless we can show that A-plonk-B is a function of A and B, i.e., given A and B, there 
is only one proposition A-plonk-B.
\end{quote} 
A possible, and natural, interpretation of what it means for a symbol $\Q$ to be uniquely definable is the following: If $\Q'$ is another symbol governed by the same inference rules as $\Q$, then $\varphi \vdash \varphi'$; where $\Q'$ is not mentioned in $\varphi$ and $\varphi'$ is like $\varphi$ but with all occurrences of $\Q$ replaced by $\Q'$. 

In \cite{feferman2012quantifiers} Feferman argues for a model theoretic take on this proof theoretic explanation of logicality. In model theoretic terms, conservativity can be understood as existence of a sound interpretation of the symbol\footnote{This is not a perfect translation; even when given a sound and complete proof system, model theoretic conservativity (i.e., the existence of a sound interpretation) is strictly stronger than proof theoretic conservativity.} and the uniqueness can be understood as uniqueness of that interpretation. Thus, the proof theoretic notion of unique definability corresponds, in some sense, to the model theoretic notion of \emph{implicit definability}.

In this short paper I will give a new elementary proof of the main result of Feferman in \cite{feferman2012quantifiers} that implicitly definable quantifiers (in a second-order language) are exactly the (explicitly) definable ones (in a first-order language). We end with some open questions and remarks. 

\section*{Technical preliminaries}

Throughout the paper $M$ will denote the domain of a model $\str M$.

We will focus on quantifier symbols, i.e., symbols that bind variables. Their model theoretic interpretations are generalized quantifiers:
\begin{defin}[\cite{lindstrom1966first}]
A \emph{generalized quantifier} $Q$ of type $\langle n_1,n_2,\ldots,n_k \rangle$ is a (class) function mapping sets to sets such that
$$Q_M=Q(M) \subseteq \mathcal P(M^{n_1}) \times \mathcal P(M^{n_2}) \times \ldots \times \mathcal P(M^{n_k}).$$
\end{defin}

The syntax and semantics of the extension of first-order logic with a quantifier symbol $\Q$ of type $\langle n_1,n_2,\ldots,n_k \rangle$, $L(\Q)$, can be exemplified by the formula
$$\Q \bar x_1, \bar x_2,\ldots ,\bar x_k (\varphi_1, \varphi_2, \ldots,\varphi_k)$$
together with the truth condition corresponding to an interpretation $Q$ of the symbol $\Q$:
$$\str M,\!s \models \Q\bar x_1, \ldots, \bar x_k (\varphi_1, \ldots,\varphi_k) \text{ iff } (\llbracket\bar x_1 | \varphi_1\rrbracket^{\str M,s}, \ldots, \llbracket\bar x_k | \varphi_k\rrbracket^{\str M,s}) \in Q_M,$$
where $\llbracket\bar y | \varphi\rrbracket^{\str M,s} = \set{\bar a \in M^l | \str M,s[\bar a / \bar y] \models \varphi}$, and $l$ is the length of the tuple $\bar y$.

The quantifiers of first-order logic, $\forall$ and $\exists$ are examples of generalized quantifiers:
\begin{itemize}
\item $\forall_M = \set{M}$
\item $\exists_M = \set{A \subseteq M | A \neq \emptyset}$
\end{itemize}

Other examples of generalized quantifiers include the following cardinality quantifiers:
\begin{itemize}
\item $(Q_0)_M = \set{A \subseteq M | |A| \geq \aleph_0}$, and in general
\item $(Q_\alpha)_M = \set{A \subseteq M | |A| \geq \aleph_\alpha}$,
\item $I_M = \set{\langle A,B \rangle | A,B \subseteq M, |A|=|B|}$
\end{itemize}

\subsection*{Second-order implicit definitions}

If we agree that the model theoretic notion of implicit definability in some way is the right model theoretic translation of Belnap's proof theoretic notion of unique definability, then  we need to know what it means for a generalized quantifier to be implicitly definable. To answer that we are going to introduce a very basic second-order language: 

The second-order language $L_2$ is pure second-order logic with two kinds of variables:
\begin{itemize}
\item Individual variables: $x,y,z,\ldots$, and
\item Predicate variables (including 0-ary) $P, P_1, P_2 \ldots$
\end{itemize}
The atomic formulas are predicate variables applied to individual variables, as in $P(x,y)$. Formulas are built from atomic formulas using $\lnot,\lor,\land,\imp,\forall,\exists$.
Observe that we may quantify using both individual variables, as in $\forall y$, and predicate variables, as in $\exists P$.

The semantics for $L_2$ used in \cite{feferman2012quantifiers} is so-called Henkin semantics:  An $L_2$ model consists of a set $M$ together with subsets $\Pred_k(\str M)$ of $\mathcal P(M^k)$ for the predicate variables to range over. We follow Feferman in that the subsets $\Pred_k(\str M)$ may be chosen in a completely arbitrary way.\footnote{Observe that in the rather restricted language of $L_2$, there are no first-order formulas (a formula without second-order quantifiers) without free predicate variables. Thus, a very restricted form of comprehension holds in all models for trivial reasons.}

\begin{example}Let the domain of $\str M$ be $\set{1,2}$ and $\Pred_1(\str M)=\set{ \emptyset}$, i.e., the unary predicates are only allowed to be interpreted as the empty set. Then $$\str M \models \forall P\, \forall x \, \lnot P(x).$$
\end{example}

The language $L_2(\Q)$ is $L_2$ extended with a \emph{second-order} predicate symbol $\Q$  of some specified type $\langle n_1,\ldots,n_k \rangle$. For example, $\forall P \, \Q(P)$ is a sentence of $L_2(\Q)$, if $\Q$ is of type $\langle k \rangle$ and $P$ is of arity $k$.

A model of $L_2(\Q)$ consists of a model $\str M$ of $L_2$ together with an interpretation of $\Q$ as a second-order predicate, i.e., a subset of $\Pred_{n_1}(\str M) \times \ldots \times \Pred_{n_k}(\str M)$. A generalized quantifier $Q$ of type $\langle n_1,\ldots,n_k \rangle$ defines a second-order predicate over any $L_2$ model $\str M$:
$$ Q_{\str M} = Q_M \cap \bigl(\Pred_{n_1}(\str M) \times \ldots \times \Pred_{n_k}(\str M)\bigr). $$

\begin{defin}
A sentence $\sigma$ of $L_2(\Q)$  \emph{implicitly defines} a generalized quantifier $Q$ if, for every $L_2$ model $\str M$, the only second-order predicate over $\str M$ satisfying $\sigma$ is $Q_{\str M}$.\end{defin}

Compare this with the notion of (explicit) definability:  

\begin{defin} A formula $\sigma(P_1,\ldots,P_k)$ of $L_2$ (explicitly) defines a generalized quantifier $Q$ if for every $L_2$ model $\str M$, and every $R_1 \in \Pred_{n_1}(\str M), \ldots, R_k \in \Pred_{n_k}(\str M)$:  $$(\str M,R_1,\ldots,R_k) \models \sigma(P_1,\ldots,P_k) \text{ iff } \langle R_1,\ldots, R_k\rangle \in Q_{\str M}.$$
\end{defin}

We can now state and prove a slightly weaker form of the main result of \cite{feferman2012quantifiers}:

\begin{lem}\label{lemma:feferman}
All generalized quantifiers of type $\langle n_1,\ldots,n_k \rangle$, where $n_i \neq n_j$, $i \neq j$, implicitly definable in $L_2(\Q)$ are definable in first-order logic.
\end{lem}
\begin{proof}
Let us give the full proof for the case where $k=1$ and $n_1=1$, i.e., where $Q$ is of type $\langle 1 \rangle$. The general result follows by the same argument. 

Assume $Q$ is implicitly definable in $L_2$ by the formula $\theta$,
where $\theta$ is written in prenex normal form.
For any $A \subseteq M$ let $M_A$ be the model of $L_2$ in which the first-order variables range over the set $M$, the second-order
\emph{unary} variables range over the singleton set $\set{A}$, and the polyadic second-order variables range over $\emptyset$.

Let $A \subseteq M$ be such that $A \in Q_M$ (if there are no such $M$ and $A$
then $Q$ is first-order definable). Since $\theta$ implicitly defines $Q$
we know that $Q_M \cap \Pred_1(M_A) =  \set{A}$
satisfies $\theta$, but $\emptyset$ does not. We may conclude that there cannot be
any polyadic variables occurring in $\theta$ since then $\theta$ would be true or false
(only depending on the first polyadic quantifier being existential or universal) in $M_A$
regardless of how $\Q$ is interpreted.

Let $\varphi$ be the first-order formula we get by removing the second-order
prefix from $\theta$ and replacing
all second-order variables by the single unary predicate symbol $P$. We also
replace all occurrences of $\Q(P)$ by some true sentence, for example
$\exists x Px \lor \lnot \exists x Px$.

It should be clear that for any $A \subseteq M$, $$(M,A) \models  \varphi \text{ iff } (M_A,\set{A}) \models 
\theta.$$ Note that $(M_A,X)$ is to be interpreted as the model $M_A$ in which $X$ is the interpretation of the quantifier symbol (second-order predicate symbol) $\Q$.

We know that $\theta$ implicitly defines $Q$, meaning that $(M_A,\set{A}) \models \theta$
iff $\set{A} = Q_M \cap M_A$. Clearly $\set{A} = Q_M \cap M_A$ iff $A \in Q_M$.

Thus, for any $A \subseteq M$, $$(M,A) \models  \varphi \text{ iff } A \in Q_M.$$
The first-order sentence $\varphi$ (explicitly) defines $Q$.
\end{proof}

An easy observation now gives us the full result:

\begin{thm}[\cite{feferman2012quantifiers}]\label{thm:feferman}
All generalized quantifiers implicitly definable in $L_2(\Q)$ are definable in first-order logic.
\end{thm}
\begin{proof}Given a quantifier of type $\langle n_1,\ldots,n_k \rangle$ we define 
$Q^{+l}$ of type $$\langle n_1,\ldots,n_{l-1},n_l+1,n_{l+1},\ldots, n_k \rangle$$ as follows:
$$(Q^{+l})_M = \set{\langle R_1,\ldots,R_{l-1},R_l\times M, R_{l+1},\ldots,R_k | \langle R_1,\ldots,R_k \rangle \in Q_M}.$$
It should be clear that $Q^{+l}$ is implicitly definable in $L_2$ if $Q$ is, and that $Q$ is definable in first-order logic if $Q^{+l}$ is. The theorem then follows directly from the lemma.
\end{proof}

\section*{Conclusion and Discussion}

The proof of the theorem rests on the simple idea that this general version of Henkin semantics does not allow second-order quantification to be used in a substantial way, not because of deep results but of completely elementary reasons. I believe the above argument suggests that the specific choice of Henkin semantics for $L_2$ is plainly wrong. Maybe a more natural choice would be to include some comprehension axioms for the Henkin semantics: Only consider models $\str M$ satisfying $$\forall \bar R \, \exists P\, \forall \bar x\, \bigl(\varphi(\bar x, \bar R) \leftrightarrow P(\bar x)\bigr),$$ 
where $\varphi$ is any formula of $L_2$. This leads us to a stronger notion of implicit definability and the following open question:

\begin{que}
Which quantifiers are implicitly definable in $L_2$ equipped with this stronger form of Henkin semantics?
\end{que}

Regardless of the choice of semantics to use there is clearly something funny with thinking of arbitrarily complicated formulas of $L_2$ as somehow corresponding to inference rules. 
Instead, to find the correct model theoretic take on Belnap's idea to characterize logicality in terms of unique definability, we need to analyze the concept of an \emph{inference rule}  to see what the right logic to express such rules is. A first attempt in this direction might be the observation that many inference rules can be formalized by a $\Pi_1^1$ formula. This leads us directly to the following question:

\begin{que}
Which quantifiers are implicitly definable in full second-order logic (i.e., second-order logic with standard semantics) with a $\Pi_1^1$ sentence?
\end{que}

Clearly, any first-order definable quantifier is implicitly definable with a $\Pi_1^1$ sentence. In fact, any $\Delta_1^1$ definable quantifier is. 

\begin{que}
Are there $\Delta^1_1$ definable quantifiers that are not first-order definable?
\end{que}

\bibliography{refs}
\bibliographystyle{alpha}

\end{document}